\documentclass[12pt]{amsart}
\usepackage{amsfonts,amssymb,amscd,amsmath,enumerate,verbatim}
\usepackage[latin1]{inputenc}
\usepackage{amscd}
\usepackage{latexsym}
\usepackage{mathptmx}
\usepackage{multicol}
\input xy
\xyoption{all}
\usepackage{hyperref}
\usepackage{color}
\def\NZQ{\mathbb}               

\def\ZZ{{\NZQ Z}}
\def\RR{{\NZQ R}}

\def\frk{\mathfrak}               

\def\Phi{{\frk N}}
\def\eb{{\bold e}}

\def\xb{{\bold x}}
\def\yb{{\bold y}}

\def\opn#1#2{\def#1{\operatorname{#2}}} 
\opn\chara{char} 
\opn\length{\ell} 
\opn\pd{pd} 
\opn\rk{rk}
\opn\projdim{proj\,dim} 
\opn\injdim{inj\,dim} 
\opn\rank{rank}
\opn\depth{depth} 
\opn\grade{grade} 
\opn\height{height}
\opn\embdim{emb\,dim} 
\opn\codim{codim}

\opn\Tr{Tr} 
\opn\bigrank{big\,rank}
\opn\superheight{superheight}
\opn\lcm{lcm}
\opn\trdeg{tr\,deg}
\opn\reg{reg} 
\opn\lreg{lreg} 
\opn\ini{in} 
\opn\lpd{lpd}
\opn\size{size}
\opn\mult{mult}
\opn\dist{dist}
\opn\cone{cone}
\opn\lex{lex}
\opn\rev{rev}
\opn\bipyr{bipyr}
\opn\div{div} \opn\Div{Div} \opn\cl{cl} \opn\Cl{Cl}
\opn\Spec{Spec} \opn\Supp{Supp} \opn\supp{supp} \opn\Sing{Sing}
\opn\Ass{Ass} \opn\Min{Min}
\opn\Ann{Ann} \opn\Rad{Rad} \opn\Soc{Soc}
\opn\Syz{Syz} \opn\Im{Im} \opn\Ker{Ker} \opn\Coker{Coker}
\opn\Am{Am} \opn\Hom{Hom} \opn\Tor{Tor} \opn\Ext{Ext}
\opn\End{End} \opn\Aut{Aut} \opn\id{id} \opn\ini{in}

\opn\nat{nat}
\opn\pff{pf}
\opn\Pf{Pf} \opn\GL{GL} \opn\SL{SL} \opn\mod{mod} \opn\ord{ord}
\opn\Gin{Gin}
\opn\Hilb{Hilb}\opn\adeg{adeg}\opn\std{std}\opn\ip{infpt}
\opn\Pol{Pol}
\opn\sat{sat}
\opn\Var{Var}
\opn\Gen{Gen}

\opn\aff{aff} \opn\con{conv} \opn\relint{relint} \opn\st{st}
\opn\lk{lk} \opn\cn{cn} \opn\core{core} \opn\vol{vol}
\opn\link{link} \opn\star{star}
\opn\gr{gr}

\def\Pc{{\mathcal P}}
\def\Qc{{\mathcal Q}}

\def\vol{{\textnormal{vol}}}
\def\conv{{\textnormal{conv}}}

\def\ord{{\textnormal{ord}}}
\def\pot#1#2{#1[\kern-0.28ex[#2]\kern-0.28ex]}

\opn\dirlim{\underrightarrow{\lim}}
\opn\inivlim{\underleftarrow{\lim}}
\def\Implies{\ifmmode\Longrightarrow \else
	\unskip${}\Longrightarrow{}$\ignorespaces\fi}
\def\implies{\ifmmode\Rightarrow \else
	\unskip${}\Rightarrow{}$\ignorespaces\fi}
\def\iff{\ifmmode\Longleftrightarrow \else
	\unskip${}\Longleftrightarrow{}$\ignorespaces\fi}

\let\:=\colon
\newtheorem{Theorem}{Theorem}[section]

\newtheorem{Example}[Theorem]{Example}

\newtheorem{Conjecture}[Theorem]{Conjecture}

\numberwithin{equation}{section}

\let\epsilon\varepsilon
\let\phi=\varphi
\let\kappa=\varkappa
\def\qed{\ifhmode\textqed\fi
	\ifmmode\ifinner\quad\qedsymbol\else\dispqed\fi\fi}
\def\textqed{\unskip\nobreak\penalty50
	\hskip2em\hbox{}\nobreak\hfil\qedsymbol
	\parfillskip=0pt \finalhyphendemerits=0}
\def\dispqed{\rlap{\qquad\qedsymbol}}

\opn\dis{dis}
\opn\height{height}
\opn\dist{dist}
\def\pnt{{\raise0.5mm\hbox{\large\bf.}}}

\opn\Lex{Lex}
\opn\conv{conv}
\opn\Ehr{Ehr}

\begin{document}
\title{Odd cycles and Hilbert functions of their toric rings}
	\author[T.~Hibi]{Takayuki Hibi}
\address[Takayuki Hibi]{Department of Pure and Applied Mathematics,
	Graduate School of Information Science and Technology,
	Osaka University,
	Suita, Osaka 565-0871, Japan}
\email{hibi@math.sci.osaka-u.ac.jp}
\author{Akiyoshi Tsuchiya}
\address[Akiyoshi Tsuchiya]
{Graduate school of Mathematical Sciences,
University of Tokyo,
Komaba, Meguro-ku, Tokyo 153-8914, Japan} 
\email{akiyoshi@ms.u-tokyo.ac.jp}

\subjclass[2010]{Primary 13A02; Secondary 13H10}
\keywords{$O$-sequence, $h$-vector, flawless, toric ring, stable set polytope}
\thanks{The first author was partially supported by JSPS KAKENHI 19H00637.  The second author was partially supported by JSPS KAKENHI 19K14505 and 19J00312.}
\begin{abstract}
Studying Hilbert functions of concrete examples of normal toric rings, it is demonstrated that, for each $1 \leq s \leq 5$, an $O$-sequence $(h_0, h_1, \ldots, h_{2s-1}) \in \ZZ_{\geq 0}^{2s}$ satisfying the properties that (i) $h_0 \leq h_1 \leq \cdots \leq h_{s-1}$, (ii) $h_{2s-1} = h_0$, $h_{2s-2} = h_1$ and (iii) $h_{2s - 1 - i} = h_i + (-1)^{i}$, $2 \leq i \leq s - 1$, can be the $h$-vector of a Cohen--Macaulay standard $G$-domain.
\end{abstract}
\maketitle
\thispagestyle{empty}
\section*{Background}
In the paper \cite{flawless} published in 1989, several conjectures on Hilbert functions of Cohen--Macaulay integral domains are studied.

Let $A = \bigoplus_{n=0}^{\infty} A_n$ be a {\em standard} $G$-algebra \cite{Sta}.  Thus $A$ is a Noetherian commutative graded ring for which (i) $A_0 = K$ a field, (ii) $A = K[A_1]$ and (iii) $\dim_K A_1 < \infty$.  The {\em Hilbert function} of $A$ is defined by
\[
H(A,n) = \dim_K A_n, \, \, \, \, \, n = 0, 1, 2, \ldots
\]
Let $\dim A = d$ and $v = H(A,1) = \dim_K A_1$.  A classical result (\cite[Chapter 5, Section 13]{Mat} says that $H(A,n)$ is a polynomial for $n$ sufficiently large and its degree is $d - 1$.  It follows that the sequence $h(A) = (h_0, h_1,  h_2, \ldots)$, called the {\em $h$-vector} of $A$, defined by the formula 
\[
(1 - \lambda)^d \sum_{n=0}^{\infty} H(A,n) \lambda^n = \sum_{i=0}^{\infty} h_i \lambda^i
\]
has finitely many non-zero terms with $h_0 = 1$ and $h_1 = v - d$.  If $h_i =0$ for $i > s$ and $h_s \neq 0$, then we write $h(A) = (h_0, h_1, \ldots, h_s)$. 

Let $Y_1,\ldots,Y_s$ be indeterminates. A non-empty set $M$ of monomials $Y_1^{a_1} \cdots Y_r^{a_r}$ in the variables $Y_1,\ldots,Y_r$ is said to be an \textit{order ideal of monomials} if, whenever $m \in M$ and $m'$ divides $m$, then $m' \in M$. Equivalently, if $Y_1^{a_1} \cdots Y_r^{a_r} \in M$ and $0 \leq b_i \leq a_i$, then $Y_1^{b_1} \cdots Y_r^{b_r} \in M$. In particular, since $M$ is non-empty, $1 \in M$. A finite sequence $(h_0,h_1,\ldots,h_s)$ of non-negative integers is said to be an \textit{$O$-sequence} if there exists an order ideal $M$ of monomials in $Y_,\ldots,Y_r$ with each $\deg Y_i=$ such that $h_j=|\{m \in M | \deg m=j\}|$ for any $0 \leq j \leq s$. In particular, $h_0=1$.
 If $A$ is Cohen--Macaulay, then $h(A) = (h_0, h_1, \ldots, h_s)$ is an $O$-sequence \cite[p.~60]{Sta}.  Furthermore, a finite sequence $(h_0, h_1, \ldots, h_s)$ of integers with $h_0 = 1$ and $h_s \neq 0$ is the $h$-vector of a Cohen--Macaulay standard $G$-algebra if and only if $(h_0, h_1, \ldots, h_s)$ is an $O$-sequence \cite[Corollary 3.11]{Sta}. 

An $O$-sequence $(h_0, h_1, \ldots, h_s)$ with $h_s \neq 0$ is called {\em flawless} \cite[p.~245]{flawless} if (i) $h_i \leq h_{s - i}$ for $0 \leq i \leq [s/2]$ and (ii) $h_0 \leq h_1 \leq \cdots \leq h_{[s/2]}$.  A standard $G$-domain is a standard $G$-algebra which is an integral domain.  It was conjectured (\cite[Conjecture 1.4]{flawless}) that the $h$-vector of a Cohen--Macaulay standard $G$-domain is flawless.  Niesi and Robbiano \cite[Example 2.4]{NR} succeeded in constructing a Cohen--Macaulay standard $G$-domain with $(1,3,5,4,4,1)$ its $h$-vector.  Thus, in general, the $h$-vector of a Cohen--Macaulay standard $G$-domain cannot be flawless.  

In the present paper, it is shown that, for each $1 \leq s \leq 5$, an $O$-sequence 
\[
(h_0, h_1, \ldots, h_{s-1}, h_{s}, \ldots, h_{2s-2}, h_{2s-1}) \in \ZZ_{\geq 0}^{2s}
\]
satisfying the properties that
\begin{itemize}
\item[(i)]
$h_0 \leq h_1 \leq \cdots \leq h_{s-1}$,
\item[(ii)]
$h_{2s-1} = h_0$, \, $h_{2s-2} = h_1$,
\item[(iii)]
$h_{2s - 1 - i} = h_i + (-1)^{i}$, $2 \leq i \leq s - 1$
\end{itemize}
can be the $h$-vector of a normal toric rings arising from a cycle of odd length.  In particular, the above $O$-sequence, which is non-flawless for each of $s = 4$ and $s = 5$, can be the $h$-vector of a Cohen--Macaulay standard $G$-domain.

\section{Toric rings arising from odd cycles}
Let $C_{2s+1}$ denote a cycle of length $2s+1$, where $s \geq 1$, on $[2s+1] = \{1, 2, \ldots, 2s+1\}$ with the edges
\begin{eqnarray}
\label{edge}
\{1,2\}, \{2,3\}, \ldots, \{2s-1,2s\}, \{2s,2s+1\},\{2s+1,1\}.
\end{eqnarray}
A finite set $W \subset [2s+1]$ is called {\em stable} in $C_{2s+1}$ if none of the sets of (\ref{edge}) is a subset of $W$.  In particular, the empty set $\emptyset$ and $\{ 1 \}, \{ 2 \}, \ldots, \{ 2s + 1 \}$ are stable.  Let $S = K[x_1, \ldots, x_{2s+1}, y]$ denote the polynomial ring in $2s+2$ variables over $K$.  The {\em toric ring} of $C_{2s+1}$ is the subring $K[C_{2s+1}]$ of $S$ which is generated by those squarefree monomials $(\prod_{i \in W} x_i) y$ for which $W \subset [2s+1]$ is stable in $C_{2s+1}$.  It follows that $K[C_{2s+1}]$ can be a standard $G$-algebra with each $\deg (\prod_{i \in W} x_i) y = 1$.  It is shown \cite[Theorem 8.1]{cycle} that $K[C_{2s+1}]$ is normal.  In particular, $K[C_{2s+1}]$ is a Cohen--Macaulay standard $G$-domain. Now, we discuss when $K[C_{2s+1}]$ is Gorenstein. Here a Cohen--Macaulay ring is called \textit{Gorenstein} if it has finite injective dimension.

\begin{Theorem}
\label{Gor}
The toric ring $K[C_{2s+1}]$ is Gorenstein if and only if either $s = 1$ or $s = 2$.
\end{Theorem}

\begin{proof}
Since the $h$-vector of $K[C_3]$ is $(1,1)$ and since the $h$-vector of $K[C_5]$ is $(1,6,6,1)$, it follows from \cite[Theorem 4.4]{Sta} that each of $K[C_3]$ and $K[C_5]$ is Gorenstein.

Now, we show that $K[C_{2s+1}]$ is not Gorenstein if $s \geq 3$.  Let $s \geq 3$.  Write $\Qc_{C_{2s+1}} \subset \RR^{2s+1}$ for the {\em stable set polytope} of $C_{2s+1}$.  Thus $\Qc_{C_{2s+1}}$ is the convex hull of the finite set 
\[
\left\{\sum_{i \in W} \eb_i : W \mbox{ is a stable set of } G  \right\} \subset \RR^{2s+1},
\]
where $\eb_1,\ldots,\eb_{2s+1} \in \RR^{2s+1}$ are the canonical unit coordinate vectors of $\RR^{2s+1}$ and where $\sum_{i \in \emptyset} \eb_i = (0, \ldots, 0) \in \RR^{2s+1}$.
One has $\dim \Qc_{2s+1}=2s+1$.  Then \cite[Theorem 4]{stableineq} says that $\Qc_{C_{2s+1}}$ is defined by the following inequalities:
\begin{itemize}
	\item $0 \leq x_i \leq 1$ for all $1 \leq i \leq 2s+1$;
	\item $x_i+x_{i+1} \leq 1$ for all $1 \leq i \leq 2s$\,;
	\item $x_1+x_{2s+1} \leq 1$;
	\item $x_1+\cdots+x_{2s+1} \leq s$.
\end{itemize}
It then follows that each of $\Qc_{C_{2s+1}}$ and $2 \Qc_{C_{2s+1}}$ has no interior lattice points and that $(1,\ldots,1)$ is an interior lattice point of $3\Qc_{C_{2s+1}}$.
Furthermore, \cite[Theorem 4.2]{stablefacet} guarantees that the inequality 
\[
x_1+\cdots+x_{2s+1} \leq s
\]
defines a facet of $\Qc_{C_{2s+1}}$.
Let $\Pc_s=3\Qc_{C_{2s+1}}-(1,\ldots,1)$.
Thus the origin of $\RR^{2s+1}$ is an interior lattice point of $\Pc_s$ and the inequality 
\[
x_1+\cdots+x_{2s+1} \leq s-1
\]
defines a facet of $\Pc_s$.
This fact together with \cite{dual} implies that $\Pc_s$ is not {\em reflexive}.  In other words, the {\em dual polytope} $\Pc_s^{\vee}$ of $\Pc_s$ defined by
\[
\Pc_s^\vee=\{\yb \in \RR^{2s+1}  :  \langle \xb,\yb \rangle \leq 1 \ \text{for all}\  \xb \in \Pc_s \}
\]
is not a lattice polytope, where $\langle \xb,\yb \rangle$ is the usual inner product of $\RR^{2s+1}$.
It then follows from \cite[Theorem (1.1)]{DeNegriHibi} (and also from \cite[Theorem 8.1]{cycle}) that $K[C_{2s+1}]$ is not Gorenstein, as desired.
\end{proof}

It is known \cite[Theorem 4.4]{Sta} that  a  Cohen--Macaulay standard $G$-domain $A$ is Gorenstein  if and only if the $h$-vector $h(A)=(h_0,\ldots,h_s)$ is symmetric, i.e., $h_i=h_{s-i}$ for $0 \leq i \leq [s/2]$. 
Hence the $h$-vector of the toric ring $K[C_{2s+1}]$ is not symmetric when $s \geq 3$.                 

\begin{Example}
{\em
By using Normaliz \cite{Normaliz},
the $h$-vector of the toric ring $K[C_7]$ is $(1, 21, 84, 85, 21, 1)$. 
}
\end{Example}

\section{Non-flawless $O$-sequences of normal toric rings}
We now come to concrete examples of non-flawless $O$-sequences which can be the $h$-vectors of normal toric rings. 

\begin{Example}
{\em
%
The $h$-vector of the toric ring $K[C_9]$ is 
\[
(1, 66, 744, 2305, 2304, 745, 66, 1).
\]
Furthermore, 
\[
(1, 187, 5049, 37247, 96448, 96449, 37246, 5050, 187, 1)
\]
is the $h$-vector of the toric ring $K[C_{11}]$.
}
\end{Example}

We conclude the present paper with the following

\begin{Conjecture}
{\em
The $h$-vector of the toric ring $K[C_{2s+1}]$ of $C_{2s+1}$ is of the form
\[
(1, h_1, h_2, h_3, \ldots, h_i, \ldots, h_{s-1}, h_{s-1} + (-1)^{s-1}, \ldots, h_i + (-1)^{i}, \ldots, h_3 - 1, h_2 + 1, h_1, 1).
\]
}
\end{Conjecture}


\begin{thebibliography}{99}
\bibitem{Normaliz}
	W. Bruns, B. Ichim, T. R\"{o}mer, R. Sieg and C. S\"{o}ger, Normaliz, Algorithms for rational cones and affine monoids, Available at \url{https://www.normaliz.uni-osnabrueck.de}.




\bibitem{stablefacet}
V. Chv\'{a}tal,
On certain polytopes associated with graphs,
{\em J. Combin. Theory Ser. B},
{\bf 18} (1975), 138--154.

\bibitem{DeNegriHibi} E.~De~Negri and T.~Hibi. Gorenstein algebras of Veronese type. \emph{J. Algebra}, {\bf 193}(1997), 629--639.

\bibitem{cycle}
A.~Engstr\"{o}m and P. Nor\'{e}n,
Ideals of Graphs Homomorphisms,
{\em Ann. Comb.} {\bf 17} (2013), 71--103.

\bibitem{flawless}
T.~Hibi, Flawless $O$-sequences and Hilbert functions of Cohen-Macaulay integral domains, {\em J. of Pure and Appl. Algebra} {\bf 60} (1989), 245--251. 

\bibitem{dual}
T. Hibi, Dual polytopes of rational convex polytopes, 
\textit{Combinatorica} {\bf 12} (1992), 237--240. 

\bibitem{stableineq}
A. R. Mahjoub,
On the stable set polytope of a series-parallel graph,
{\em Math. Programming}, {\bf 40} (1988), 53--57.

\bibitem{Mat}
H.~Matsumura, ``Commutative Ring Theory,'' Cambridge University Press, Cambridge, 1989.
\bibitem{NR}
G.~Niesi and L.~Robbiano, Disproving Hibi's Conjecture with CoCoA or Projective Curves with bad Hilbert Functions, {\em in} ``Computational Algebraic Geometry'' (F.~Eyssette and A.~Galligo, Eds.), Birkh\"auser, Boston, MA, 1993, pp.~195--201.
\bibitem{Sta}
R.~P.~Stanley, Hilbert functions of graded algebras, {\em Advances in Math.} 
{\bf 28} (1978), 57--83.
\end{thebibliography}
\end{document}